\documentclass[12pt,oneside]{article}

\usepackage{amsmath}
\usepackage{amssymb}
\usepackage{color}
\usepackage{rotating}

\usepackage{graphicx}
\usepackage{subfigure}

\usepackage{url}

\author{Hui Zhou,
\\ \footnotesize{School of Mathematical Sciences, Peking University, Beijing, 100871, P.~R.~China.}
\\ \scriptsize{School of Mathematics and Statistics, Lanzhou University, Lanzhou, Gansu 730000, P. R. China}
\\ \footnotesize{zhouhpku17@pku.edu.cn, huizhou@math.pku.edu.cn, zhouhlzu06@126.com.}
\vspace{2em}
\\ Liufeng Xu, Yang Cui,
\\ \scriptsize{School of Mathematics and Statistics, Lanzhou University, Lanzhou, Gansu 730000, P. R. China}
\\ \footnotesize{798245503@qq.com, 360048115@qq.com}
\vspace{2em}
\\ Qi Ding,
\\ \scriptsize{School of Mathematics and Statistics, Lanzhou University, Lanzhou, Gansu 730000, P. R. China}
\\ \scriptsize{Yonyou Network Technology Co., Ltd., Beijing 100094, China}
\\ \footnotesize{892127976@qq.com, dingqi0@yonyou.com}
\vspace{2em}
\\ Yanfeng Luo, Xing Gao and Dong Yang
\\ \scriptsize{School of Mathematics and Statistics, Lanzhou University, Lanzhou, Gansu 730000, P. R. China}
\\ \footnotesize{luoyf@lzu.edu.cn, gaoxing05@lzu.edu.cn, gaoxing@lzu.edu.cn}
}

\title{\Large Hamiltonian decomposition of the Cayley graph on the dihedral group $D_{2p}$ where $p$ is a prime~\footnote{This work is supported by a project funded by China Postdoctoral Science Foundation, grant no. 2017M620491.
}}

\def\D{{\sf D}}

\def\gcd{{\rm gcd}}

\newtheorem{theorem}{Theorem}

\newtheorem{lemma}[theorem]{Lemma}

\newtheorem{definition}[theorem]{Definition}

\newcommand*{\QEDA}{\hfill\ensuremath{\blacksquare}}  
\newenvironment{proof}[1][\hspace{2ex}\textbf{\textit{Proof}.}\hspace{1ex}]{\begin{trivlist}\item[\hskip \labelsep {\bfseries #1}]}{\QEDA\end{trivlist}}

\begin{document}

\maketitle

\begin{abstract}
In this note, we give the Hamiltonian decomposition of the Cayley graph on the dihedral group $D_{2p}$ where $p$ is a prime.

\end{abstract}





\section{Introduction}\label{sec Introduction}

Let $G$ be a group and let $S$ be a subset of $G$. Let $S^{-1}=\{s^{-1}\mid s\in S\}$ and let $\langle S\rangle$ be the subgroup of $G$ generated by $S$. If $S=\{a\}$, then we use $\langle a\rangle$ to denote $\langle S\rangle$. We suppose $1\not\in S$, $S=S^{-1}$ and $\langle S\rangle=G$. The Cayley graph on $G$ with respect to $S$ is a graph $\Gamma=Cay(G,S)$ such that the vertex-set is $G$ and two vertices $g\in G$ and $h\in G$ are adjacent in $\Gamma$ if and only if $hg^{-1}\in S$, i.e., there exists $s\in S$ such that $h=sg$ (this edge is also called an $s$-edge or $s^{-1}$-edge). Let $T$ be a nonempty set of $S$ such that $T=T^{-1}$. An edge is called a $T$-edge if there exist $t\in T$ such that this edge is a $t$-edge. The Cayley graph $\Gamma=Cay(G,S)$ is $G$-vertex-transitive, $G$-regular, and it is a regular graph of valency $|S|$.
A Hamilton cycle of a graph is a cycle of length the order of the graph. 
Alspach~\cite{Alspach1984Researchproblem59} stated the problem of Hamilton decompositions on Cayley graphs.

\begin{definition}[Hamilton decomposition]\label{def Hamilton ecomposition}
Let $k\geqslant 2$. A regular graph of valency $2k$ is said to have a Hamilton decomposition if its edge-set can be partitioned into $k$ Hamilton cycles. A regular graph of valency $2k-1$ is said to have a Hamilton decomposition if its edge-set can be partitioned into $k-1$ Hamilton cycles and a perfect matching.
\end{definition}

It was conjectured that every connected Cayley graph on an abelian group has a Hamilton decomposition. Many researchers have studied on this conjecture. 
For nonabelian groups, the dihedral groups come out first. Here we consider Hamilton decompositions of Cayley graphs on dihedral groups.
Let $n\geqslant 2$. The dihedral group $\D_{2n}$ is the group generated by two elements $\alpha$ and $\beta$ satisfying the relations \begin{center}$\alpha^n=\beta^2=1$ and $\beta\alpha\beta=\alpha^{-1}$.\end{center} 

Alspach~\cite{Alspach1979vt2pH} proved that, with the exception of the Petersen graph, every connected vertex-transitive graph of order $2p$, where $p$ is a prime, has a Hamilton cycle. Since the Petersen graph is not a Cayley graph, Alspach's result implies Holszty\'{n}ski and Strube's result~\cite[Proposition~6.1]{HolsztynskiStrube1978ConjectureDihedralgroupH}.

\begin{theorem}\cite{HolsztynskiStrube1978ConjectureDihedralgroupH}\label{thm D2p is H}
Let $p$ be a prime. Then every connected Cayley graph on $\D_{2p}$ has a Hamilton cycle.
\end{theorem}

Alspach~\cite{Alspach1980vt2pHD} showed that every connected vertex-transitive graph of order $2p$, where $p$ is a prime and $p\equiv 3\pmod 4$, has a Hamilton decomposition. In this note, we mainly give Hamilton decompositions of Cayley graphs on $\D_{2p}$ where $p$ is a prime. More precisely, we prove the following theorem.

\begin{theorem}\label{thm D2p is HD}
Let $p$ be a prime. Then every connected Cayley graph on $\D_{2p}$ has a Hamilton decomposition.
\end{theorem}


\section{Proof of Theorem~\ref{thm D2p is HD}}\label{sec proof D2p is HD}

Let $A$ and $B$ be two set. We use $A\Delta B$ to denote the symmetric difference of $A$ and $B$. For two graphs $\Sigma$ and $\Gamma$, the symbol $\Sigma\Delta\Gamma$ means the symmetric difference of the edge sets of graphs $\Sigma$ and $\Gamma$.

We first give results on two special cases.

\begin{lemma}\label{lem one involution}
Let $p\geqslant 3$ be an odd prime and let $\Gamma=Cay(\D_{2p},S)$ such that $S\subseteq \D_{2p}$, $1\not\in S$, $S=S^{-1}$ and $\langle S\rangle=\D_{2p}$. Let $S_2=S\cap \beta\langle\alpha\rangle$. If $|S_2|=1$, then $\Gamma$ has a Hamilton decomposition.
\end{lemma}

\begin{proof}
Suppose $|S_2|=1$. Since $S=S^{-1}$, then $|S\cap \langle\alpha\rangle|$ is even. So we can write $S=\{\alpha^{i_1},\alpha^{i_2},\ldots,\alpha^{i_s},\alpha^{p-i_s},\ldots,\alpha^{p-i_2},\alpha^{p-i_1},b\}$ where $S_2=\{b\}$, $s\geqslant 1$ and $1\leqslant i_1<i_2<\cdots<i_s<p-i_s<\cdots<p-i_2<p-i_1\leqslant p-1$. We know $s+i_s\leqslant p-1$, and the $b$-edges form a perfect matching $M$ of $\Gamma$. Let $1\leqslant t\leqslant s$ and $R_t=\{\alpha^{i_t},\alpha^{p-i_t}\}$. Then $\langle R_t\rangle=\langle\alpha\rangle$. The $R_t$-edges in $\Gamma$ form two cycles: one cycle $C_{t,\alpha}$ of length $p$ on $\langle\alpha\rangle$ and the other cycle $C_{t,\beta}$ of length $p$ on $\beta\langle\alpha\rangle$. Consider the cycle $C_{t}=(\alpha^t,\alpha^{t+i_t},b\alpha^{t+i_t},b\alpha^t)$ of length four. Then $H_t=(C_{t,\alpha}\cup C_{t,\beta})\Delta C_t$ is a Hamilton cycle of $\Gamma$ and $M\Delta C_t$ is also a perfect matching of $\Gamma$. Note that the cycles $C_1,C_2,\ldots,C_t$ are vertex-disjoint. Hence $M_C=M\Delta\left(\bigcup\limits_{j=1}^sC_j\right)$ is also a perfect matching. The edges of $\Gamma$ is partitioned into Hamilton cycles $H_1,H_2,\ldots,H_s$ and a perfect matching $M_C$, i.e., $\Gamma$ has a Hamilton decomposition.
\end{proof}

\begin{lemma}\label{lem all involution}
Let $p\geqslant 3$ be an odd prime and let $\Gamma=Cay(\D_{2p},S)$ such that $S\subseteq \D_{2p}$, $1\not\in S$, $S=S^{-1}$ and $\langle S\rangle=\D_{2p}$. Let $S_1=S\cap \langle\alpha\rangle$. If $|S_1|=0$, then $\Gamma$ has a Hamilton decomposition.
\end{lemma}

\begin{proof}
Suppose $|S_1|=0$. Then $S\subseteq \beta\langle\alpha\rangle$, and so we can write the set $S=\{\beta\alpha^{i_1},\beta\alpha^{i_2},\ldots,\beta\alpha^{i_s}\}$ where $s\geqslant 2$ and $0\leqslant i_1<i_2<\cdots<i_s\leqslant p-1$. Let $1\leqslant t\leqslant \lfloor\frac{s}{2}\rfloor$ and let $Q_t=\{\beta\alpha^{i_{2t-1}},\beta\alpha^{i_{2t}}\}$. Then $\D_{2p}=\langle Q_t\rangle$, and so the $Q_t$-edges form a Hamilton cycle $C_t$ of $\Gamma$. Note that the $\beta\alpha^{i_s}$-edges form a perfect matching $M$ of $\Gamma$. If $s$ is even, then the edges of $\Gamma$ is partitioned into Hamilton cycles $C_1,C_2,\ldots,C_t$. If $s$ is odd, then the edges of $\Gamma$ is partitioned into Hamilton cycles $C_1,C_2,\ldots,C_t$ and a perfect matching $M$. Thus $\Gamma$ has a Hamilton decomposition.
\end{proof}


Now we consider the Hamilton decomposition of a special tetravalent Cayley graph on the dihedral group $\D_{2n}$.

\begin{lemma}\label{lem tetravalent}
Let $n\geqslant 3$. Let $\Gamma=Cay(\D_{2n},S)$ where $S=\{\alpha^i,\alpha^{-i},\beta\alpha^j,\beta\alpha^k\}$, $1\leqslant i\leqslant n-1$, $1\leqslant j<k\leqslant n-1$ and $\gcd(i,n)=\gcd(k-j,n)=1$. Then $\Gamma$ has a Hamilton decomposition.
\end{lemma}

\begin{proof}
Let $a=\alpha^i$, $b=\beta\alpha^j$. Then $\D_{2n}=\langle a,b\rangle$. Since $\gcd(i,n)=1$, the congruence equation $ix\equiv k-j\pmod n$ has a solution $x=s$ with $0\leqslant s\leqslant n-1$. Note that $\gcd(k-j,n)=1$, we have $\gcd(s,n)=1$. Then $\beta\alpha^k=ba^s$. We have $S=\{a,a^{-1},b,ba^s\}$ and $\Gamma=Cay(\D_{2n}=\langle a,b\rangle,S=\{a,a^{-1},b,ba^s\})$.

The $\{a,a^{-1}\}$-edges in $\Gamma$ form one cycle $C_{a,\alpha}$ of length $n$ on $\langle\alpha\rangle$ and the other cycle $C_{a,\beta}$ of length $n$ on $\beta\langle\alpha\rangle$. Consider the cycle $C_{a}=(a^{-1},1,b,ba^{-1})$ of length four. Then $H_a=(C_{a,\alpha}\cup C_{a,\beta})\Delta C_a$ is a Hamilton cycle of $\Gamma$. Let $H_b$ be obtained from $\Gamma$ by deleting edges of $H_a$. The edges of $H_b$ are as follows: $\{a^k,ba^k\}$ for $1\leqslant k\leqslant n-2$, $\{a^k,ba^{k+s}\}$ for $0\leqslant k\leqslant n-1$, $\{1, a^{-1}\}$ and $\{b,ba^{-1}\}$. Then $H_b$ is a regular graph of degree two. So we only need to show $H_b$ is connected.

Since $\gcd(s,n)=1$, there exist integers $u$ and $v$ such that $us+vn=1$. Then $n\nmid u$, and so we can write $-u=qn+r$ where $q$ is an integer and $0<r<n$. We have $a^{-1}=a^{vn-1}=a^{-us}=a^{(qn+r)s}=a^{rs}$, i.e., $a^{rs+1}=1$. Let $t=n-r$. Then $0<t<n$ and $a=a^{1+sn}=a^{1+s(r+t)}=a^{1+rs}a^{ts}=a^{ts}$, i.e., $a^{ts-1}=1$. Since $\gcd(s,n)=1$, we know $r=\min\{m>0\mid a^{ms+1}=1\}$ and $t=\min\{m>0\mid a^{ms-1}=1\}$.

By the edges in $H_b$, we know $1$ is connected to \begin{center}$1,ba^s,a^s,ba^{2s},a^{2s},\ldots,ba^{(r-1)s},a^{(r-1)s},ba^{rs}=ba^{-1}$,\end{center} and $a^{-1}$ is connected to \begin{center}$a^{-1},ba^{s-1},a^{s-1},ba^{2s-1},a^{2s-1},\ldots,ba^{(t-1)s-1},a^{(t-1)s-1},ba^{ts-1}=b$.\end{center} These vertices are exactly all the vertices of $\Gamma$. Since $\{1, a^{-1}\}$ and $\{b,ba^{-1}\}$ are two edges in $H_b$, we get that $H_b$ is connected and so it is a Hamilton cycle of $\Gamma$. Thus edges of $\Gamma$ is partitioned into two Hamilton cycles $H_a$ and $H_b$, i.e., $\Gamma$ has a Hamilton decomposition.
\end{proof}


\begin{lemma}\label{lem two involutions}
Let $p\geqslant 3$ be an odd prime and let $\Gamma=Cay(\D_{2p},S)$ such that $S\subseteq \D_{2p}$, $1\not\in S$, $S=S^{-1}$ and $\langle S\rangle=\D_{2p}$. Let $S_2=S\cap \beta\langle\alpha\rangle$. If $|S_2|=2$, then $\Gamma$ has a Hamilton decomposition.
\end{lemma}

\begin{proof}
We suppose $|S_2|=2$. Let $S_1=S\cap \langle\alpha\rangle$. Then $|S_1|$ is even. If $|S_1|=0$, then $\Gamma$ itself is a Hamilton cycle. If $|S_1|=2$, then by Lemma~\ref{lem tetravalent}, $\Gamma$ has a Hamilton decomposition. In the following we assume $|S_1|\geqslant 4$.

Let $S_2=\{b,c\}$ and let $a=bc$. Then $\D_{2p}=\langle S_2\rangle$ and $\langle a\rangle =\langle\alpha\rangle$. We may write $S_1=\{a^{i_1},a^{i_2},\ldots,a^{i_s},a^{p-i_s},\ldots,a^{p-i_2},a^{p-i_1}\}$ where $s\geqslant 2$ and $1\leqslant i_1<i_2<\cdots<i_s<p-i_s<\cdots<p-i_2<p-i_1\leqslant p-1$. We know $s+i_s\leqslant p-1$, and the $S_2$-edges form a Hamilton cycle $C$ of $\Gamma$. Let $1\leqslant t\leqslant s$ and $R_t=\{a^{i_t},a^{p-i_t}\}$. Then $\langle R_t\rangle=\langle\alpha\rangle$. The $R_t$-edges in $\Gamma$ form two cycles: one cycle $C_{t,\alpha}$ of length $p$ on $\langle\alpha\rangle$ and the other cycle $C_{t,\beta}$ of length $p$ on $\beta\langle\alpha\rangle$. Consider the cycle $C_{t}=(a^{m_t},a^{m_t+i_t},ba^{m_t+i_t},ba^{m_t})$ of length four where $0\leqslant m_t\leqslant p-1$. Then $H_t=(C_{t,\alpha}\cup C_{t,\beta})\Delta C_t$ is a Hamilton cycle of $\Gamma$. Note that the cycles $C_1,C_2,\ldots,C_t$ are vertex-disjoint. Let
\begin{equation*}
D=C\Delta\left(\bigcup\limits_{j=1}^sC_j\right).
\end{equation*}
Then $D$ is regular of degree two. The edges of $\Gamma$ is partitioned into Hamilton cycles $H_1,H_2,\ldots,H_s$ and an edge set $D$. If $D$ is also a Hamilton cycle of $\Gamma$, then $\Gamma$ has a Hamilton decomposition. In the following, we will choose suitable $m_1,m_2,\ldots,m_s$ such that $D$ is connected.

It is easy to see that $C\Delta C_t$ is a Hamilton cycle for $1\leqslant t\leqslant s$. 
Let $1\leqslant t\leqslant s$ and $1\leqslant r\leqslant s$ with $t\neq r$ satisfying one of the following conditions
\begin{equation}\label{eqn disjoint}
0\leqslant m_t<m_t+i_t<m_r<m_r+i_r\leqslant p-1,
\end{equation}
or
\begin{equation}\label{eqn cover}
0\leqslant m_t<m_r<m_r+i_r<m_t+i_t\leqslant p-1.
\end{equation}
Then $C_t$ and $C_r$ are vertex-disjoint, and so $C\Delta (C_t\cup C_r)=C\Delta C_t\Delta C_r$ is also a Hamilton cycle. 
If we can find $m_1,m_2,\ldots,m_s$ such that any two cycles $C_t$ and $C_r$ satisfying one of the two conditions as in Equations~(\ref{eqn disjoint}) and (\ref{eqn cover}), then $D$ is a Hamilton cycle of $\Gamma$. Now we give a suitable choice of $m_1,m_2,\ldots,m_s$ with this property. 

Note that $2\leqslant s\leqslant \frac{p-1}{2}$ and $i_j+2\leqslant i_{j+2}$ for $1\leqslant j<s$. Let $A=\{0,1,2,\ldots,\frac{p-1}{2}\}$ and $B=\{p-1,p-2,\ldots,\frac{p+1}{2}\}$. Then $A$ and $B$ form a partition of $\{0,1,2,\ldots,p-1\}$. Since $i_s\leqslant \frac{p-1}{2}=|A|-1$ and $i_{s-1}\leqslant \frac{p-3}{2}=|B|-1$, we may let $m_s\in A$ and $m_{s-1}\in B$. If $s=2k$ where $k\geqslant 1$, let
\begin{enumerate}
\item $m_{2h}=k-h$ for $1\leqslant h\leqslant k$ and
\item $m_{2h-1}=\frac{p+1}{2}+(k-h)$ for $1\leqslant h\leqslant k$,
\end{enumerate}
then we have
\begin{eqnarray*}
&&0\leqslant m_s<m_{s-2}<\cdots< m_{2h}<\cdots< m_2<\\
&&  m_2+i_2<\cdots< m_{2h}+i_{2h}<\cdots<m_{s-2}+i_{s-2}<m_s+i_s\leqslant \frac{p-1}{2}\text{ and }\\
&&\frac{p+1}{2}\leqslant m_{s-1}<m_{s-3}<\cdots< m_{2h-1}<\cdots<m_1<\\
&&  m_1+i_1<\cdots< m_{2h-1}+i_{2h-1}<\cdots<m_{s-3}+i_{s-3}<m_{s-1}+i_{s-1}\leqslant p-1,
\end{eqnarray*}
and so any pair of the cycles $C_1,C_2,\ldots,C_s$ satisfying one of the conditions in Equations~(\ref{eqn disjoint}) and (\ref{eqn cover}), which implies $D$ is a Hamilton cycle of $\Gamma$. If $s=2k+1$ where $k\geqslant 1$, let
\begin{enumerate}
\item $m_{2h+1}=k-h$ for $0\leqslant h\leqslant k$ and
\item $m_{2h}=\frac{p+1}{2}+(k-h)$ for $1\leqslant h\leqslant k$,
\end{enumerate}
then we have
\begin{eqnarray*}
&&0\leqslant m_s<m_{s-2}<\cdots< m_{2h+1}<\cdots< m_1<\\
&&  m_1+i_1<\cdots< m_{2h+1}+i_{2h+1}<\cdots<m_{s-2}+i_{s-2}<m_s+i_s\leqslant \frac{p-1}{2}\text{ and }\\
&&\frac{p+1}{2}\leqslant m_{s-1}<m_{s-3}<\cdots< m_{2h}<\cdots<m_2<\\
&&  m_2+i_2<\cdots< m_{2h}+i_{2h}<\cdots<m_{s-3}+i_{s-3}<m_{s-1}+i_{s-1}\leqslant p-1,
\end{eqnarray*}
and so any pair of the cycles $C_1,C_2,\ldots,C_s$ satisfying one of the conditions in Equations~(\ref{eqn disjoint}) and (\ref{eqn cover}), which implies $D$ is a Hamilton cycle of $\Gamma$. We complete the proof.
\end{proof}






Now we give the proof of Theorem~\ref{thm D2p is HD}.

\begin{proof}[Proof of Theorem~\ref{thm D2p is HD}.]
The case $p=2$ is easy to get. In the following we only consider odd prime $p\geqslant 3$. Let $p\geqslant 3$ be an odd prime and let $\Gamma=Cay(\D_{2p},S)$ such that $S\subseteq \D_{2p}$, $1\not\in S$, $S=S^{-1}$ and $\langle S\rangle=\D_{2p}$. We need to show that $\Gamma$ has a Hamilton decomposition. Let $S_1=S\cap \langle\alpha\rangle$ and $S_2=S\cap \beta\langle\alpha\rangle$. Then $|S_1|$ is even and $|S_2|\geqslant 1$. By Lemma~\ref{lem one involution}, Lemma~\ref{lem all involution} and Lemma~\ref{lem two involutions}, we may assume $|S_1|\geqslant 2$ and $r=|S_2|\geqslant 3$.

First we suppose $r$ is even. Let $A$ and $B$ be a partition of $S_2$ such that $|A|=2$ and $|B|=r-2$. Let $T=S_1\cup A$. Then $T$ and $B$ form a partition of $S$. By Lemma~\ref{lem two involutions}, the $T$-edges is partitioned into Hamilton cycles. Note that $|B|$ is even, so the $B$-edges is partitioned into Hamilton cycles. Hence $\Gamma$ has a Hamilton decomposition.

Now we let $r$ be odd. Take $a\in S_2$. Let $P=S_1\cup \{a\}$ and let $R=S_2\setminus \{a\}$. Then $P$ and $R$ form a partition of $S$. By Lemma~\ref{lem one involution}, the $P$-edges is partitioned into Hamilton cycles and a perfect matching. Note that $|R|$ is even, so the $R$-edges is partitioned into Hamilton cycles. Hence $\Gamma$ has a Hamilton decomposition.
\end{proof}


%



%
%

\end{document}